\documentclass[leqno]{amsart}

\usepackage{a4,latexsym,amssymb,amsfonts,amsmath,mathrsfs,bm}

\newcommand{\R}{{\mathbb R}}

\newcommand{\Z}{{\mathbb Z}}
\newcommand{\eps}{{\varepsilon}}

\newcommand{\xx}{{\textbf{x}}}
\newcommand{\yy}{{\textbf{y}}}
\newcommand{\zz}{{\textbf{z}}}
\newcommand{\hh}{{\textbf{h}}}
\newcommand{\vv}{{\textbf{v}}}

\newcommand{\bb}{{\textbf{b}}}
\newcommand{\cc}{{\textbf{c}}}
\newcommand{\M}{{\mathfrak{M}}}
\newcommand{\m}{{\mathfrak{m}}}
\newcommand{\dalpha}{{d\bm{\alpha}}}
\newcommand{\dbeta}{{d\bm{\beta}}}
\newcommand{\deta}{{d\bm{\eta}}}
\newcommand{\dv}{{d\textbf{v}}}
\newcommand{\dy}{{d\textbf{y}}}

\newtheorem{theorem}{Theorem}[section]

\newtheorem{lemma}[theorem]{Lemma}
\theoremstyle{definition}

\theoremstyle{remark}

\begin{document}

\title{On the representation of quadratic forms by quadratic forms}

\author{Rainer Dietmann}
\address{Department of Mathematics, Royal Holloway, University of London,
TW20 OEX Egham, UK}
\curraddr{}
\email{Rainer.Dietmann@rhul.ac.uk}

\author{Michael Harvey}
\address{Department of Mathematics, Royal Holloway, University of London,
TW20 OEX Egham, UK}
\email{Michael.Harvey@rhul.ac.uk}

\subjclass[2000]{11D09, 11D72, 11E12, 11P55}

\date{}

\begin{abstract}
Using the circle method, we show that for a fixed positive definite
integral
quadratic form $A$, the expected asymptotic formula for the number of
representations of a positive definite integral quadratic
form $B$ by $A$ holds true, providing that the dimension
of $A$ is large enough in terms of the dimension of $B$ and the
maximum ratio of the successive minima of $B$, and providing that
$B$ is sufficiently large in terms of $A$.
\end{abstract}

\maketitle

\bibliographystyle{amsplain}
\bibliography{siegel3}

\section{Introduction}

The study of representing an integral quadratic form by another integral
quadratic form has a long history in number theory. In this paper we
use matrix notation for quadratic forms, so
let $A = (A_{ij})$ and $B=(B_{ij})$ be symmetric positive definite integer
matrices, of dimensions $n$ and $m$, respectively.
We are interested in finding $n \times m$ integer matrices $X$ such that
\begin{equation}\label{xax}X^{T}AX = B,\end{equation}
this way generalising the classical problem of representing a positive
integer as a sum of squares.
While the Local-Global principle is known to hold true for
\textit{rational} solutions $X$ of the Diophantine problem
\eqref{xax}, existence of solutions over $\mathbb{R}$ (here automatic
by positive definiteness) and all local
rings $\mathbb{Z}_p$ is not enough to ensure existence of an
\textit{integer} solution $X$. 
It is therefore natural to look for additional conditions
making sure that the Local-Global principle also holds over $\mathbb{Z}$.
The usual point of view then is to fix $m$,
$n$ and $A$ and try to represent
`large enough' $B$ for dimension $m$ as large as possible in terms of $n$.
In this context,
Hsia, Kitaoka and Kneser \cite{HsiaKitaokaKneser} have shown the
Local-Global principle to hold true, whenever $n \geq 2m+3$ and
$\min B \geq c_{1}$ for some constant $c_{1}$ depending only on
$A$ and $n$, where as usual $\min B$ denotes the first
successive minimum of $B$, i.e.
\[
  \min B=\min_{\mathbf{x} \in \mathbb{Z}^m \backslash \{\mathbf{0}\}}
  \mathbf{x}^T B \mathbf{x}.
\]
Recent work of Ellenberg and Venkatesh \cite{EllenbergVenkatesh} used ergodic theory to show that the condition on
$n$ can be greatly improved to $n \geq m+5$, 
under the additional assumption that the discriminant of $B$ is square-free.  This latter condition has been refined
by Schulze-Pillot \cite{Schulze-Pillot}.

The approaches above do not give any quantitative information about integer solutions to \eqref{xax}.  Let
$N(A,B)$ denote the number of integer matrices $X$ satisfying \eqref{xax}.
Note that this quantity is finite as $A$ is positive definite.
Siegel \cite{Siegel} gave an exact formula for a weighted version of $N(A,B)$.  Let
$\mathfrak{A}$ be a set of representatives of all equivalence classes of forms in 
the genus of $A$.  For such a representative $A \in \mathfrak{A}$, let $o(A)$ denote
the number of automorphs of $A$, and let $W(\mathfrak{A}) = \sum_{A \in \mathfrak{A}}1/o(A)$.  
Then Siegel showed that
$$\frac{\sum_{A \in \mathfrak{A}}N(A,B)/o(A)}{W(\mathfrak{A})} = 
\left\{\begin{array}{ll}\alpha_{\infty}(A,B)\prod_{p}\alpha_{p}(A,B)
& \mbox{if } m < n-1\\
\frac{1}{2}\alpha_{\infty}(A,B)\prod_{p}\alpha_{p}(A,B)
& \mbox{if } m = n-1,\end{array}\right.$$
where these factors depend only on the genera of $A$ and $B$, the term
\begin{equation}\label{alphainfty}
\alpha_{\infty}(A,B) = (\det{A})^{-m/2}(\det{B})^{(n-m-1)/2}\pi^{m(2n-m+1)/4}\prod_{n-m < j \leq n}(\Gamma(j/2))^{-1}
\end{equation}
corresponds to the density of real solutions to \eqref{xax}, and for any prime $p$, we have
\begin{equation}\label{alphap}
\alpha_{p}(A,B) = (p^{-t})^{mn-m(m+1)/2}\#\{X \bmod{p^{t}}: X^{T}AX \equiv B \pmod{p^{t}}\},
\end{equation}
for all sufficiently large integers $t$.  In particular, if the genus of $A$ contains only one equivalence class,
then this gives an exact formula for $N(A,B)$, but if the genus of $A$
contains more than one class, we only get some upper bound on $N(A,B)$.
Our focus in this paper is on obtaining an \textit{asymptotic} formula
for $N(A,B)$ rather than an exact one, but valid for all forms $A$.
By asymptotic, in this context we mean asymptotic in terms of the
successive minima of $B$.
Without changing $N(A,B)$, by replacing $B$ by an equivalent form
if necessary, we may assume that $B$ is Minkowski-reduced.
In particular,
$$ 0 < \min B = B_{11} \leq B_{22} \leq \cdots \leq B_{mm},
\qquad |B_{ij}| \leq B_{ii} \;\;\; (1 \leq i < j \leq m).$$
For $1 \leq i \leq m$, define $\gamma_{i}$ to be the positive
real number which satisfies
\begin{equation}
\label{gamma_i}
  B_{11} = B_{ii}^{\gamma_{i}},
\end{equation}
and define
\begin{equation}
\label{gamma}
\gamma := \sum_{i=1}^{m}\frac{1}{\gamma_{i}}.
\end{equation}
Note that $\gamma_i \le 1 \;\;\; (1 \le i \le m)$.

\begin{theorem}\label{thm}
Suppose that $n > (2\gamma + m(m-1))\Big(\frac{m(m+1)}{2}+1\Big)$.  Then there exists $\delta > 0$ such that
\begin{equation}\label{Nabthm}
N(A,B) = \alpha_{\infty}(A,B)\prod_{p}\alpha_{p}(A,B) +
O((\det{B})^{\frac{n-m-1}{2} - \delta}),
\end{equation}
where $\alpha_{\infty}(A,B), \alpha_{p}(A,B)$ are defined above, and
the implied $O$-constant does not depend on $B$.
\end{theorem}

For $n \ge 2m+3$
it was shown by Kitaoka (see \cite{Kitaoka}, Proposition 9) that
\begin{equation}
\label{ungef}
  1 \ll \prod_{p}\alpha_{p}(A,B) \ll 1,
\end{equation}
whenever \eqref{xax} is soluble over each $\mathbb{Z}_{p}$,
with implicit constants independent of $B$.
Recalling \eqref{alphainfty}, we find that
the main term in \eqref{Nabthm} is of greater order of magnitude
than the error term, and gives a true asymptotic formula provided
that $\det{B}$ is large enough in terms of $m$, $n$ and $A$.
Since $\gamma$ is bounded for fixed $m$ and $n$, the latter condition is
equivalent to
$B_{11} \geq c_{2}$ for some constant $c_{2}$ depending only on $m$, $n$,
and $A$.

Let us now
briefly connect our result to others to be found in the literature:
When $m=1$, then as mentioned at the beginning, the equation
\eqref{xax} reduces to the classical problem
of representing a positive integer by a positive definite 
quadratic form.  
One attains such an asymptotic formula for $N(A,B)$ as long as $n \geq 3$ (see \cite{DukeSchulze-Pillot}, \cite{Heath-Brown}).  
For general $m >1$, Raghavan \cite{Raghavan} uses the theory of Siegel modular forms to establish an asymptotic formula
whenever $n \geq 2m+3$, under the assumption
\begin{equation}\label{minB}
B_{11} = \min B \geq c_{3}(\det{B})^{1/m},\end{equation}
for some fixed constant $c_{3}$.  We note that there exists some constant $c_{4}$
depending only on $m$ such that $B_{11} \leq c_{4}(\det{B})^{1/m}$.  Therefore Raghavan's result requires the
successive minima of $B$ to be of similar orders, which
essentially translates into
the condition $\gamma=1$ in our setting.
Our result does not require this condition, but may require
a much larger number $n$ of variables
when $B_{11}$ is much smaller than $\det{B}$.  
For $m=2$ and $n \ge 7$, Kitaoka \cite{Kitaoka} shows that the condition
$B_{11} \geq c_{5}$ for some constant $c_{5}$ only depending
on $n$ and $A$
suffices, avoiding any further assumptions and
this way paralleling what is known
for $m=1$. No such result is known so far for $m>2$;
see also Schulze-Pillot \cite{Schulze-Pillot2}
for more background information on this topic.

Whereas most previous approaches to this problem were using modular
forms, 
our strategy is to treat \eqref{xax} as a system of $R:= \frac{m(m+1)}{2}$ quadratic
equations, and apply the circle method; see also \cite{D} and
\cite{JB} for circle method approaches to related higher degree
problems.
The main difficulty is adapting the method to work in a 
box with uneven side lengths, and it is exactly here
where the dependence on $\gamma$ comes in.
We obtain a version of Weyl's inequality in \S 2 on
following the method of Birch \cite{Birch}
as well as an argument of Parsell (see \cite{Parsell}, Lemma 4.1).
We then use this to estimate the minor arcs in \S 3, before handling the major arcs in \S 4.  Once this has been accomplished, we need to show that we
have the main term in the Theorem by examining the singular series and
singular integral in \S\S 5 and 6.\\ \\
\textbf{Notation:}
As usual, $\eps$ will denote a small positive number that may change
value from one statement to the next.
All implied constants may depend on $A$, $m$, $n$, $\eps$.
We apply the usual notation that
$e(z) = e^{2\pi i z}$, $e_{q}(z)=e^{\frac{2\pi i z}{q}}$.
We use $\|x\|$ to denote the distance of the nearest integer to a real number $x$.  We also set
$|\xx| = \max_{1 \leq i \leq n}|x_{i}|$
for the maximum norm for any vector $\xx \in \mathbb{R}^{n}$, and we
write $(a,b)$ for the greatest common divisor of two integers $(a,b)$.
Summations over vectors $\mathbf{x}$ are usually to be understood
as summation over $\mathbf{x} \in \mathbb{Z}^n$, and
multidimensional integrations are usually to be understood to be
performed in $R$-dimensional space. We sometimes use conditions of
the form $q \ll L$ for a certain quantity $L$, in particular in
summations and integrals.
These are to be understood in the following way: There
exists a suitable constant $C$, depending at most on $A$, $m$, $n$,
$\eps$, such that the condition $q \ll L$ can be replaced by
$q \le CL$.
\\ \\
\textbf{Acknowledgements:} 
This work has been supported by grant EP/I018824/1 `Forms in many
variables', funding a one year PostDoc position
at Royal Holloway for the second author.

\section{Weyl-type inequalities}
By letting $X= (\xx_{1} \cdots \xx_{m})$, with column vectors
$\xx_{i} =(x_{i1},\ldots,x_{in}) \in \mathbb{Z}^{n}$ for each
$i \in \{1, \ldots, m\}$, we may write
\eqref{xax} as the following system of
\[
  R:= \frac{m(m+1)}{2}
\]
equations:
$$\xx_{i}^{T}A\xx_{j} = B_{ij} \quad (1 \leq i \leq j \leq m).$$
Clearly, since $A$ is positive definite, these equations imply that
\[
  |\mathbf{x}_i| \ll B_{ii}^{1/2} \quad (1 \le i \le m)
\]
for an implicit constant depending only on $A$, and since $B$ is
positive definite, there exists a real
solution of these equations within that range. Therefore, for
sufficiently large $C$ depending only on $A$, define
\begin{equation}
\label{holiday}
  P_{i} := C^{1/\gamma_i} B_{ii}^{1/2}
\end{equation}
for each $i \in \{1, \ldots, m\}$, and note that by \eqref{gamma_i},
we have
\begin{equation}\label{Piprop}
0 < P_{1} \leq \cdots \leq P_{m}, \qquad P_{1} = P_{i}^{\gamma_{i}} \quad
 (1 \leq i \leq m).
\end{equation}
For convenience, we shall also define
$$\Pi := \prod_{i=1}^{m}P_{i}.$$
Note that
\begin{equation}
\label{gamma_useful}
  \Pi = P_i^{\gamma \gamma_i} \quad (1 \le i \le m)
\end{equation}
by \eqref{gamma} and \eqref{Piprop}.

For real $\bm{\alpha} = (\alpha_{ij})_{1 \leq i \leq j \leq m}$ and $\bb = (B_{ij})_{1 \leq i \leq j \leq m}$,
we define the exponential sum
$$S(\bm{\alpha}, \bb) := \sum_{|\xx_{1}| \leq P_{1}}\cdots\sum_{|\xx_{m}| \leq P_{m}}
e\Big(\sum_{1 \leq i \leq j \leq m}\alpha_{ij}(\xx_{i}^{T}A\xx_{j} - B_{ij})\Big),$$
and let $S(\bm{\alpha}) := S(\bm{\alpha}, \textbf{0})$.
By our choice of $C$ and the $P_i$  we then have
\begin{equation}\label{NAB}
N(A,B) = \int_{[0,1)^{R}}S(\bm{\alpha}, \bb)\dalpha.
\end{equation}

Our aim is to show that, as long as $n$ is large enough, then $S(\bm{\alpha})$ is `small', unless each $\alpha_{ij}$ is well-approximated by a rational number with small denominator. The next lemma achieves this for the diagonal
coefficients $\alpha_{ii}$.

\begin{lemma}\label{lemmaii}
Let $0 < \theta < 1$.  Suppose that $S(\bm{\alpha}) \gg \Pi^{n-k}$ for some positive real number $k$.  
For each $i \in \{1, \ldots, m\}$, if 
\begin{equation}\label{ncond}
n > \frac{2k\gamma_{i}\gamma}{\theta},
\end{equation}
then there exists an integer $q_{ii} \geq 1$ satisfying
$$q_{ii} \ll P_{i}^{\theta}, \qquad \|q_{ii}\alpha_{ii}\| \leq P_{i}^{-2 + \theta}.$$
\end{lemma}

\begin{proof}
Fix $i \in \{1, \ldots, m\}$.  Then we have
\begin{equation}
\label{salphaleq}
  |S(\bm{\alpha})| \leq
  \sum_{|\xx_{1}| \leq P_{1}}\cdots\sum_{|\xx_{i-1}| \leq P_{i-1}}\sum_{|\xx_{i+1}| \leq P_{i+1}}\cdots\sum_{|\xx_{m}| \leq P_{m}}|T_{i}(\bm{\alpha})|,
\end{equation}
where 
\begin{equation}\label{Tidef}
T_{i}(\bm{\alpha}) = T_{i}(\bm{\alpha};\xx_{1},\ldots,\xx_{i-1},\xx_{i+1},\ldots,\xx_{m}) := 
\sum_{|\xx_{i}|\leq P_{i}}e\Big(\sum_{j=1}^{m}\alpha_{ij}\xx_{i}^{T}A\xx_{j}\Big),\end{equation}
and for ease of notation we let $\alpha_{ij} = \alpha_{ji}$ if $i > j$. 
By squaring and differencing, it is clear that on writing
$\mathbf{z}=\tilde{\mathbf{x}}_i-\mathbf{x}_i$, we obtain
\begin{align}
\label{weyl_step}
  |T_i(\bm{\alpha})|^2 = & \sum_{|\mathbf{x}_i| \le P_i}
  \sum_{|\tilde{\mathbf{x}}_i| \le P_i} e \left( \sum_{\substack{j=1 \\ j \ne i}}^m
  \alpha_{ij} (\tilde{\mathbf{x}}_i^T-\mathbf{x}_i^T) A \mathbf{x}_j
  + \alpha_{ii} (\tilde{\mathbf{x}}_i^T A \tilde{\mathbf{x}}_i -
  \mathbf{x}_i^T A \mathbf{x}_i) \right) \nonumber \\
  = & \sum_{|\mathbf{z}| \le 2P_i}
  \sum_{\substack{|\mathbf{x}_i| \le P_i:\\ |\mathbf{z}+\mathbf{x}_i| \le P_i}}
  e \left( \sum_{\substack{j=1\\j \ne i}}^m \alpha_{ij}
  \mathbf{z}^T A \mathbf{x}_j + \alpha_{ii} (\mathbf{z}^T A \mathbf{z}
  + 2 \mathbf{x}_i^T A \mathbf{z}) \right).
\end{align}
In particular,
\begin{align*}
|T_{i}(\bm{\alpha})|^{2} \leq&
\sum_{|\zz| \leq 2P_{i}}\Big|\sum_{\substack{|\xx| \leq P_{i}:\\|\zz+\xx| \leq P_{i}}}
e(2\alpha_{ii}\xx^{T}A\zz)\Big|\\
\ll & 
\sum_{|\zz| \leq 2P_{i}}\prod_{u=1}^{n}\min\{P_{i},\|2\alpha_{ii}(A_{u1}z_{1} + \cdots + A_{un}z_{n}) \|^{-1}\},
\end{align*}
uniformly in $\xx_{1},\ldots,\xx_{i-1},\xx_{i+1},\ldots,\xx_{m}$.

Let
\begin{align*}
  N(\alpha_{ii}, P_{i}) := \# & \{\zz \in \Z^n: |\zz|\leq P_{i}
  \mbox{ and } \\
  & \|2\alpha_{ii}(A_{u1}z_{1} + \cdots + A_{un}z_{n})\| \leq P_{i}^{-1} \quad
(1 \leq u \leq n)\}.
\end{align*}
Then by standard techniques (see the proof of Lemma 13.2 in Davenport \cite{Davenportbook} for instance), for any $\eps >0$ we have
$$|T_{i}(\bm{\alpha})|^{2} \ll N(\alpha_{ii}, P_{i})P_{i}^{n+\eps}.$$

For any real number $\theta$ with $0 < \theta < 1$ define
\begin{align*}
  M(\alpha_{ii}, P_{i}^{\theta}) := \# & \{\zz \in \Z^n:
 |\zz|\leq P_{i}^{\theta} \mbox{ and } \\
  & \|2\alpha_{ii}(A_{u1}z_{1} + \cdots + A_{un}z_{n})\| \leq P_{i}^{-2+\theta} \quad (1 \leq u \leq n)\}.
\end{align*}
Then we have
$$M(\alpha_{ii}, P_{i}^{\theta}) \gg P_{i}^{n\theta - n}N(\alpha_{ii}, P_{i}),$$
by a standard argument using Davenport \cite[Lemma 12.6]{Davenportbook},
as in the proof of \cite[Lemma 13.3]{Davenportbook}. Therefore
$$|T_{i}(\bm{\alpha})|^{2} \ll P_{i}^{2n - n\theta + \eps}M(\alpha_{ii}, P_{i}^{\theta}),$$
and hence, by \eqref{salphaleq}, we conclude that
$$S(\bm{\alpha}) \ll \Pi^{n}P_{i}^{- n\theta/2 + \eps}M(\alpha_{ii}, P_{i}^{\theta})^{1/2}.$$

Suppose that $S(\bm{\alpha}) \gg \Pi^{n-k}$ for some positive real number
$k$.  Then we have
$$\Pi^{n-k} \ll S(\bm{\alpha}) \ll 
\Pi^{n}P_{i}^{- n\theta/2 + \eps}M(\alpha_{ii}, P_{i}^{\theta})^{1/2},$$
and thus
\begin{align*}
M(\alpha_{ii}, P_{i}^{\theta}) &\gg \Pi^{-2k}P_{i}^{n\theta - \eps}\\
& = P_{i}^{-2k\gamma_{i}\gamma + n\theta - \eps},\end{align*}
on using (\ref{gamma_useful}).

Our assumption \eqref{ncond} implies that this exponent is strictly positive for small enough $\eps$, and
therefore we have $M(\alpha_{ii}, P_{i}^{\theta}) \geq 2$.
Hence there exists some $\zz \in \Z^n$ such that
$\zz \neq \textbf{0}$ and
$$|\zz| \leq P_{i}^{\theta}, \qquad \|2\alpha_{ii}(A_{u1}z_{1} + \cdots + A_{un}z_{n})\| \leq P_{i}^{-2+\theta} \quad
(1 \leq u \leq n).$$
Now, since $\zz$ is non-zero and our matrix $A$ is non-singular, we have
$$A_{u1}z_{1} + \cdots + A_{un}z_{n} \neq 0$$
for some $u \in \{1, \ldots, n\}$.
For this $u$, define $q_{i} = 2|A_{u1}z_{1} + \cdots + A_{un}z_{n}| \neq 0$.
Then we have
$1 \le q_{i} \ll P_{i}^{\theta}$, and $\|q_{i}\alpha_{ii}\| \leq P_{i}^{-2+\theta}.$
\end{proof}

We deal with the remaining $\alpha_{ij}$ where $i \neq j$
in the following lemma, whose proof is along the lines of 
\cite[Lemma 4.1]{Parsell}.
\begin{lemma}\label{lemmaij}
Let $\delta$ be a real number satisfying $0 < \delta \leq \frac{1}{\gamma}$.  
Suppose that $S(\bm{\alpha}) \gg \Pi^{n-k}$ for some positive real number $k$.
For fixed $i,j$ satisfying $1 \leq i < j \leq m$, suppose that
\begin{equation}\label{ncond2}
n > 2k\gamma_i \gamma.
\end{equation}
Then there exists an integer $q_{ij} \geq 1$ such that
$$q_{ij} \ll \Pi^{\frac{2k}{n}+\delta}, 
\quad \|q_{ij}\alpha_{ij}\| \leq \Pi^{\frac{2k}{n}+\delta}(P_{i}P_{j})^{-1}.$$
\end{lemma}

\begin{proof}
Fix $i, j \in \{1, \ldots, m\}$ such that
$1 \leq i < j \leq m$. By an application of the Cauchy--Schwarz inequality, we have
\begin{equation}\label{Sboundl2}|S(\bm{\alpha})|^{2} \le (\Pi P_{i}^{-1})^{n}
\sum_{\substack{|\xx_{t}| \leq P_{t}\\(1 \leq t \leq m, t \neq i,j)}}\sum_{|\xx_{j}| \leq P_{j}}
|T_{i}(\bm{\alpha})|^{2},\end{equation}
where $T_{i}(\bm{\alpha})$ has been defined in \eqref{Tidef}.
Now \eqref{weyl_step} gives
\begin{align}\label{Tsquare}\nonumber
\sum_{|\xx_{j}| \leq P_{j}}|T_{i}(\bm{\alpha})|^{2}&\leq
\sum_{|\yy| \leq P_{i}}\sum_{\substack{|\hh| \le 2P_i:\\|\yy + \hh| \leq P_{i}}}
\Big|\sum_{|\xx_{j}| \leq P_{j}}e(\alpha_{ij}\hh^{T}A\xx_{j})\Big|\\\nonumber
& \ll \sum_{|\yy| \leq P_{i}}\sum_{\substack{|\hh| \le 2P_i:\\|\yy + \hh| \leq P_{i}}}\prod_{u=1}^{n}
\min\{P_{j}, \|\alpha_{ij}(A_{u1}h_{1} + \cdots + A_{un}h_{n})\|^{-1}\}\\
& \ll (P_{i})^{n}\sum_{|\hh| \leq 2P_{i}}\prod_{u=1}^{n}\min\{P_{j}, \|\alpha_{ij}(A_{u1}h_{1} + 
\cdots + A_{un}h_{n})\|^{-1}\}.
\end{align}

Let
\[
  z_{u} = A_{u1}h_{1} + \cdots + A_{un}h_{n}
\]
for each
$u \in \{1, \ldots, n\}$, and set 
\[
  \lambda := 2nP_{i} \max_{1 \le i,j \le n} |A_{ij}|.
\]
Note that
\[
  P_i \ll \lambda \ll P_i.
\]
Then, since $A$ is a non-singular matrix, we have
\begin{align*}
& \sum_{|\hh| \leq 2P_{i}}\prod_{u=1}^{n}\min\{P_{j}, \|
\alpha_{ij}(A_{u1}h_{1} + 
\cdots + A_{un}h_{n})\|^{-1}\} \\ \ll &
\sum_{\substack{|\zz| \leq \lambda}}
\prod_{u=1}^{n}\min\{P_{j}, \|\alpha_{ij}z_{u}\|^{-1}\}\\
= & \Big(\sum_{|z| \leq \lambda}
\min\{P_{j},\|\alpha_{ij}z\|^{-1}\}\Big)^{n}.
\end{align*}

Combining \eqref{Sboundl2} and \eqref{Tsquare}, we have
\begin{align*}
|S(\bm{\alpha})|^{2} &\ll \Pi^{2n}(P_{i}P_{j})^{-n}
\Big(  P_j +
\sum_{1 \le z \leq \lambda}\min\{P_{j},\|\alpha_{ij}z\|^{-1}\}\Big)^{n}\\
& \ll \Pi^{2n} \left( P_i^{-n} + \left(
\frac{1}{q} + \frac{1}{P_{j}} + \frac{q}{P_{i}P_{j}}\right)^{n}
(\log{2P_{i}q})^{n} \right), \end{align*}
on applying Vaughan \cite[Lemma 2.2]{Vaughanbook} 
provided that $|\alpha_{ij} - \frac{a}{q}| \leq q^{-2}$ for coprime integers $a,q$ with $q \geq 1$.

Let $\delta$ be a positive real number with $0 < \delta \leq \frac{1}{\gamma}$.  Then we use
\eqref{gamma}, \eqref{Piprop} and \eqref{gamma_useful} to see that
$$P_{i}P_{j}\Pi^{-\frac{2k}{n}-\delta} = P_{j}^{1 + \frac{\gamma_{j}}{\gamma_{i}} -
\frac{2k\gamma_{j}\gamma}{n}-\delta\gamma_{j}\gamma} > 1,$$
using $\gamma_i \ge \gamma_j$ and \eqref{ncond2}.
By Dirichlet's theorem, there exist coprime integers $q_{ij}, a$ satisfying
$$1 \leq q_{ij} \leq P_{i}P_{j}\Pi^{-\frac{2k}{n}-\delta}, \quad 
|q_{ij}\alpha_{ij}-a| \leq (P_{i}P_{j})^{-1}\Pi^{\frac{2k}{n}+\delta}.$$  
Therefore 
$$S(\bm{\alpha})\ll \Pi^{n+\eps}
\left( P_i^{-n/2} + \left(\frac{1}{q_{ij}} + \frac{1}{P_{j}} +
\frac{q_{ij}}{P_{i}P_{j}}\right)^{n/2} \right).$$

Now
$$\Pi^{n+\eps}P_{j}^{-n/2} \ll \Pi^{n-k},$$
provided that $P_{j} \gg \Pi^{\frac{2k}{n}+\eps}$.
For sufficiently small $\varepsilon$,
this follows from
\eqref{ncond2} because of
$$\Pi^{\frac{2k}{n}+\eps} = P_{j}^{\frac{2k\gamma_{j}\gamma}{n} + \eps}$$
by \eqref{gamma_useful}.
Analogously, we get
\[
  \Pi^{n+\eps} P_i^{-n/2} \ll \Pi^{n-k}.
\]
Since $q_{ij} \leq P_{i}P_{j}\Pi^{-\frac{2k}{n}-\delta}$, it follows that
$$\Pi^{n+\eps}\Big(\frac{q_{ij}}{P_{i}P_{j}}\Big)^{n/2} \ll \Pi^{n-k},$$
provided $\eps >0$ is small enough compared to $\delta$.
Therefore we obtain
$$S(\bm{\alpha})\ll \frac{\Pi^{n+\eps}}{q_{ij}^{n/2}} + \Pi^{n-k}.$$
If $q_{ij} \gg \Pi^{\frac{2k}{n}+\delta}$, then we have $S(\bm{\alpha})\ll \Pi^{n-k}$
for sufficiently small $\eps >0$,
which contradicts the hypothesis of the lemma.  Hence we have $q_{ij} \ll \Pi^{\frac{2k}{n}+\delta}$.
\end{proof}

Combining the previous two lemmas gives the following Weyl-type inequality for our exponential sum $S(\bm{\alpha})$.
\begin{lemma}\label{Weyl}
Let $0 < \theta < 1$ and $k > 0$.  Assume that 
\begin{equation}\label{ncond3}n > \frac{2k\gamma}{\theta}.\end{equation} 
Then either (i) we have
\[
  S(\bm{\alpha}) \ll \Pi^{n-k},
\]
or (ii), there exist integers $q, a_{ij} \quad
(1 \leq i \leq j \leq m)$ that are coprime,\\
i.e. $(q, \mathbf{a})=(q, a_{11}, a_{12}, \ldots, a_{mm})=1$, such that
\begin{gather}
  1 \le q \ll \Pi^{\theta(1+\frac{m(m-1)}{2\gamma})}, \nonumber\\
\label{approx}
  |q\alpha_{ij} - a_{ij}| \ll \Pi^{\theta(1+\frac{m(m-1)}{2\gamma})} (P_{i}P_{j})^{-1} 
  \quad (1 \leq i \leq j \leq m).
\end{gather}
\end{lemma}

\begin{proof}
Suppose that (i) does not hold.
Since  \eqref{ncond3} implies \eqref{ncond} and \eqref{ncond2} for all
$i, j \in \{1, \ldots, m\}$, we may apply
lemmata \ref{lemmaii} and \ref{lemmaij} to show that there exist
integers $q_{ij}, b_{ij} \quad (1 \leq i \leq j \leq m)$
satisfying
\begin{gather*}
  (q_{ij},b_{ij}) = 1,\\
  q_{ii} \ll P_{i}^{\theta},
  \quad |q_{ii}\alpha_{ii} - b_{ii}| \leq P_{i}^{\theta -2},\\
  q_{ij} \ll \Pi^{\frac{2k}{n}+\delta}, \quad |q_{ij}\alpha_{ij} - b_{ij}| \leq
  \Pi^{\frac{2k}{n}+\delta}(P_{i}P_{j})^{-1} \quad (1 \leq i < j \leq m),
\end{gather*}
whenever $0 < \delta \leq \frac{1}{\gamma}$.  
The condition \eqref{ncond3} implies that
\[
  \frac{2k}{n}+\delta < \frac{\theta}{\gamma},
\]
provided we choose $\delta$ to be a sufficiently small positive real number.

Define $q$ to be the least common multiple of the $q_{ij} \quad
(1 \leq i \leq j \leq m)$, and set $a_{ij} := \frac{qb_{ij}}{q_{ij}}$.
Then $q$ and the $a_{ij}$ are coprime: Let $p$ be a prime dividing $q$,
and let $p^r$ be the maximum power of $p$ dividing at least one of the
$q_{ij}$. By definition of $q$, we have $r \ge 1$.
Then if $p^r \, || \, q_{ij}$, then $p$ does not divide
$\frac{q}{q_{ij}}$. Since $(q_{ij}, b_{ij})=1$, then $p$ also does not
divide $b_{ij}$, whence $p$ does not divide $a_{ij}$.
Moreover,
\begin{align*}
  q &\leq \prod_{k=1}^m q_{kk}
  \prod_{1 \le k < l \le m} q_{kl} \ll
   \Pi^{\theta(1+\frac{m(m-1)}{2\gamma})},\\
  \frac{q}{q_{ii}} &\ll (\Pi P_i^{-1})^\theta
  \Pi^{\frac{\theta m(m-1)}{2\gamma}} \quad (1 \le i \le m),\\
  \frac{q}{q_{ij}} &\ll \Pi^\theta \Pi^{\frac{\theta}{\gamma}
  (\frac{m(m-1)}{2}-1)} \quad (1 \le i<j \le m),
\end{align*}
whence
\begin{align*}
  |q\alpha_{ii}-a_{ii}| = & \frac{q}{q_{ii}} |q_{ii}\alpha_{ii}-b_{ii}|
  \ll \frac{q}{q_{ii}} P_i^{\theta-2} \ll
  \Pi^{\theta(1+\frac{m(m-1)}{2\gamma})} P_i^{-2} \quad (1 \le i \le m),\\
  |q\alpha_{ij}-a_{ij}| = & \frac{q}{q_{ij}} |q_{ij}\alpha_{ij}-b_{ij}| \ll
  \frac{q}{q_{ij}} \Pi^{\theta/\gamma} (P_i P_j)^{-1} \\ \ll &
  \Pi^{\theta(1+\frac{m(m-1)}{2\gamma})} (P_i P_j)^{-1} \quad (1 \le i<j \le m),
\end{align*}
so the bound \eqref{approx} holds true for all $i, j \in \{1, \ldots, m\}$
such that $i \le j$.
\end{proof}

\section{minor arcs}

We are now in a position to set up the scene for an application of the
circle method, splitting the $\bm{\alpha}$ into two subsets,
where either $S(\bm{\alpha})$ is small, or where each $\alpha_{ij}$ is
well-approximated.

For coprime integers $q, \textbf{a} := a_{ij} \; (1 \leq i \leq j \leq m)$,
and $\Delta > 0$, define the major arc
\begin{equation}\label{Maq}
\M_{\textbf{a},q}(\Delta) := \{\bm{\alpha} \in [0,1)^{R} : |q\alpha_{ij} - a_{ij}| \ll
\Pi^{\Delta}(P_{i}P_{j})^{-1} \quad (1 \leq i \leq j \leq m)\}.
\end{equation}
Then we define the major arcs
$\M(\Delta)$ to be the union of the $\M_{\textbf{a},q}(\Delta)$
over all coprime integers 
$q, \textbf{a}$ such that $1 \le q \ll \Pi^{\Delta}$
and $1 \le a_{ij}<q \quad (1 \leq i \leq j \leq m)$.  We denote
the minor arcs by $\m(\Delta) := [0,1)^{R}\setminus \M(\Delta)$.

We may split the integral in \eqref{NAB} to see that
\begin{equation}\label{NABsplit}
N(A,B) = \int_{\M(\Delta)}S(\bm{\alpha}, \bb)\dalpha + \int_{\m(\Delta)}S(\bm{\alpha},\bb)\dalpha.
\end{equation}
We shall show the latter integral does not contribute to the main term of the asymptotic formula for $N(A,B)$
using the following corollary of Lemma \ref{Weyl}.

\begin{lemma}\label{Weyl2}
Let $\eps > 0$ and $0 < \Delta < \frac{2\gamma+m(m-1)}{2\gamma}$ be real numbers.
Then either (i) the bound
$$S(\bm{\alpha}) \ll \Pi^{n-\frac{n\Delta}{2\gamma+m(m-1)}+\eps}$$
holds true,
or (ii), we have $$\bm{\alpha} \in \M(\Delta).$$
\end{lemma}

\begin{proof}
This follows by taking $k = \frac{n\theta}{2\gamma} - \eps$ in Lemma \ref{Weyl}, 
for $\theta = \frac{2\gamma\Delta}{2\gamma + m(m-1)}$, noting that \eqref{ncond3} is therefore
satisfied. 
\end{proof}

\begin{lemma}\label{minorarcs}
Suppose that $n > (2\gamma+m(m-1))(R+1)$.
Then for any $0 < \Delta < \frac{m+1}{R+1}$, we have
$$\int_{\m(\Delta)}S(\bm{\alpha},\bb)\dalpha \ll \Pi^{n-m-1-\delta},$$
for some $\delta > 0$.
\end{lemma}

\begin{proof}
We follow the method of Davenport and Birch, see for example
\S4 in \cite{Birch}.
Let $\delta > 0$ be a real number satisfying
\begin{equation}\label{deltadef}
\frac{n}{(2\gamma + m(m-1))} - (R+1) >\frac{2\delta}{\Delta},
\end{equation}
whose existence is guaranteed by the condition on $n$.
Define a sequence $\Delta_{0},\Delta_{1},\ldots,\Delta_{T}$ such that 
$$0 < \Delta = \Delta_{0} < \Delta_{1} < \cdots < \Delta_{T} = \frac{m+1}{R+1},$$
and has the property that
\begin{equation}\label{thetagap}
\Delta_{t+1} - \Delta_{t} < \frac{\delta}{(R+1)},
\end{equation}
for each $0 \leq t \leq T-1$. Note that
\begin{equation}
\label{induction}
  \m(\Delta) = \m(\Delta_T) \cup
  (\M(\Delta_T) \backslash \M(\Delta_{T-1}))
  \cup \ldots \cup
  (\M(\Delta_1) \backslash \M(\Delta_0)).
\end{equation}

By Lemma \ref{Weyl2}, we have for any $\eps > 0$,
\begin{align*}
\int_{\m(\Delta_{T})}|S(\bm{\alpha},\bb)|\dalpha = &
\int_{\m(\Delta_{T})}|S(\bm{\alpha})|\dalpha \\
\ll & \Pi^{n-\frac{n\Delta_{T}}{2\gamma + m(m-1)}+\eps}\\
< & \Pi^{n-\Delta_{T}(R+1+\frac{2\delta}{\Delta})+\eps}\\
< & \Pi^{n-m-1 - 2\delta + \eps},
\end{align*}
on using \eqref{deltadef}, and since $\Delta < \Delta_{T}$.  Therefore
$$\int_{\m(\Delta_{T})}|S(\bm{\alpha})|\dalpha \ll \Pi^{n-m-1 - \delta},$$
provided $\eps$ is small enough.

For $0 \leq t \leq T-1$, we have $\M(\Delta_{t+1})\setminus \M(\Delta_{t}) \subset \M(\Delta_{t+1})$, and hence
its measure is bounded by
\begin{align*}
 & \sum_{q \ll \Pi^{\Delta_{t+1}}}\sum_{\mathbf{a} (\bmod q)}
\prod_{1 \leq i \leq j
\leq m}\Big(q^{-1}\Pi^{\Delta_{t+1}}(P_{i}P_{j})^{-1}\Big)
\\ \ll & \sum_{q \ll \Pi^{\Delta_{t+1}}}\sum_{\mathbf{a}(\bmod q)}
q^{-R} \Pi^{R\Delta_{t+1}}\Pi^{-(m+1)}
\\ \ll & \Pi^{\Delta_{t+1}(R+1)- m-1}.
\end{align*}

We may therefore use Lemma \ref{Weyl2} to show that, for
sufficiently small $\eps >0$,
\begin{align*}
\int_{\M(\Delta_{t+1})
  \setminus \M(\Delta_{t})}|S(\bm{\alpha},\bb)|\dalpha
= & \int_{\M(\Delta_{t+1})
 \setminus \M(\Delta_{t})}|S(\bm{\alpha})|\dalpha\\
\ll&  
\Pi^{n-\frac{n\Delta_{t}}{2\gamma + m(m-1)}+
\Delta_{t+1}(R+1)- m-1 + \eps}\\
< & \Pi^{n-m-1 - \Delta_{t}(\frac{n}{2\gamma + m(m-1)}-(R+1)) + \delta
+ \eps},
\end{align*}
on using \eqref{thetagap}. Now $\Delta \leq \Delta_{t}$ and \eqref{deltadef}
yield
$$\int_{\M(\Delta_{t+1}) \setminus \M(\Delta_{t})}|S(\bm{\alpha})|\dalpha \ll
\Pi^{n-m-1 - \delta}.$$
By \eqref{induction},
this completes the proof, on noting that $T \ll 1$.
\end{proof}

\section{Major arcs}

In dealing with the major arcs, we shall find it more convenient
to enlarge the sets $\M_{\textbf{a},q}(\Delta)$ slightly.
Let $\M'_{\textbf{a},q}(\Delta)$ denote the set in \eqref{Maq},
but with the inequality
$$|q\alpha_{ij} - a_{ij}| \ll
q\Pi^{\Delta}(P_{i}P_{j})^{-1} \quad (1 \leq i \leq j \leq m)$$
instead, and let $\M'(\Delta)$ denote the corresponding union.

It follows from \eqref{NABsplit} and Lemma \ref{minorarcs} that, provided 
$$n > (2\gamma + m(m-1))(R+1), \qquad 0 < \Delta < \frac{m+1}{R+1},$$
we have
\begin{equation}\label{nab1}
N(A,B) = \int_{\M'(\Delta)}S(\bm{\alpha},\bb)\dalpha + O(\Pi^{n-m-1-\delta}),
\end{equation}
for some $\delta > 0$.

\begin{lemma}\label{Disjoint}
Suppose that $0 < \Delta < \frac{2}{3\gamma}$.  Then
for sufficiently large $P_1$
the major arcs $\M'_{\textbf{a},q}(\Delta)$ are disjoint.
Similarly, if $0<\Delta<\frac{1}{\gamma}$, then for sufficiently large
$P_1$ the $\M_{\textbf{a},q}(\Delta)$ are disjoint. 
\end{lemma}

\begin{proof}
Suppose that there exists $\bm{\alpha}$ lying in the intersection of two
different sets
of the form
$\M'_{\textbf{a},q}(\Delta)$. Then for some $i,j$ with
$1 \leq i \leq j \leq m$, there exist integers
$a_{ij},a'_{ij},q,q'$ with $a_{ij}q' \neq a'_{ij}q$, satisfying
\begin{align*}
q,q' \ll & \Pi^{\Delta},\\
|q\alpha_{ij} - a_{ij}| \ll & q\Pi^{\Delta}(P_{i}P_{j})^{-1},\\
|q'\alpha_{ij} - a'_{ij}| \ll & q'\Pi^{\Delta}(P_{i}P_{j})^{-1}.
\end{align*}
Hence, using \eqref{Piprop}, \eqref{gamma_useful} and
$\gamma_1=1$, we find that
\begin{align*}
1  \leq |a_{ij}q' - a'_{ij}q| = & |q'(a_{ij}-q\alpha_{ij})
+q(q'\alpha_{ij}-a'_{ij})| \\
\le & q'|q\alpha_{ij} - a_{ij}| + q|q'\alpha_{ij} - a'_{ij}|\\
\ll & \Pi^{3\Delta}(P_{i}P_{j})^{-1}\\
= & P_{1}^{3\gamma\Delta- \frac{1}{\gamma_{i}}-\frac{1}{\gamma_{j}}}.
\end{align*}
This gives a contradiction given our assumption on $\Delta$, and noting
that  $\frac{1}{\gamma_{i}}+\frac{1}{\gamma_{j}}\geq 2$.
The proof of the second statement is completely analogous.
\end{proof}

Whenever $\bm{\alpha} \in \M'_{\textbf{a},q}(\Delta)$,
we may write, for each $1 \leq i \leq j \leq m$,
\begin{equation}\label{alphaMaq}
\alpha_{ij} = \frac{a_{ij}}{q}+ \beta_{ij}, \quad |\beta_{ij}| \ll \Pi^{\Delta}(P_{i}P_{j})^{-1}
\end{equation}
for suitable integers $a_{ij}$ and $1 \le q \ll \Pi^\Delta$.
Define
$$S_{\textbf{a},q}(\textbf{b}) := \sum_{\zz_{1}\,(\bmod{q})}\cdots\sum_{\zz_{m}\,(\bmod{q})}
e_{q}\Big(\sum_{1 \leq i \leq j \leq m}a_{ij}(\zz_{i}^{T}A\zz_{j} - B_{ij})\Big),$$
and let $S_{\textbf{a},q} := S_{\textbf{a},q}(\textbf{0})$.  Also, define
$$I(\textbf{P},\bm{\beta}) := \int_{[-1,1]^{mn}}e\Big(\sum_{1 \leq i \leq j \leq m}P_{i}P_{j}\beta_{ij}\vv_{i}^{T}A\vv_{j}\Big)
\dv_{1}\cdots\dv_{m},$$ 
and let
\[
  I(\bm{\beta}) := I((1,\ldots,1),\bm{\beta}).
\]
\begin{lemma}\label{Salpha}
For $\bm{\alpha} \in \M'_{\textbf{a},q}(\Delta)$, let \eqref{alphaMaq} hold.  Then we have
$$S(\bm{\alpha}, \textbf{b}) = q^{-mn}
\Pi^n
S_{\textbf{a},q}(\textbf{b})I(\textbf{P},\bm{\beta})
e\Big(-\sum_{1 \leq i \leq j \leq m}\beta_{ij}B_{ij}\Big)
+ O\Big(\Pi^{n + 2\Delta}P_{1}^{-1}\Big).$$
\end{lemma}

\begin{proof}
Using \eqref{alphaMaq}, we have
\begin{align*}
  & S(\bm{\alpha}, \textbf{b})\\
 = & \sum_{|\xx_{1}| \leq P_{1}}\cdots\sum_{|\xx_{m}| \leq P_{m}}
e_{q}\Big(\sum_{1 \leq i \leq j \leq m}a_{ij}(\xx_{i}^{T}A\xx_{j} - B_{ij})\Big)e\Big(\sum_{1 \leq i \leq j \leq m}\beta_{ij}(\xx_{i}^{T}A\xx_{j} - B_{ij})\Big).
\end{align*}
Letting $\xx_{i} = \zz_{i}+q\yy_{i} \; (1 \leq i \leq m)$, we obtain
\begin{align*}
S(\bm{\alpha}, \textbf{b}) = & \sum_{\zz_{1},\ldots, \zz_{m}\,(\bmod{q})}
e_{q}\Big(\sum_{1 \leq i \leq j \leq m}a_{ij}(\zz_{i}^{T}A\zz_{j} - B_{ij})\Big) \\
\times &
\sum_{\yy_{1},\ldots,\yy_{m}}e\Big(\sum_{1 \leq i \leq j \leq m}\beta_{ij}((\zz_{i}+q\yy_{i})^{T}A(\zz_{j}+q\yy_{j}) - B_{ij})\Big),
\end{align*}
where the sum over $\yy_{1},\ldots,\yy_{m} \in \Z^n$
is such that $|\zz_{i} + q\yy_{i}| \leq P_{i}$ for $1 \leq i \leq m$.

It follows from Iwaniec and Kowalski \cite[Lemma 4.1]{IwaniecKowalski} and a simple induction argument that the sum 
$$\sum_{\yy_{1},\ldots,\yy_{m}}e\Big(\sum_{1 \leq i \leq j \leq m}\beta_{ij}(\zz_{i}+q\yy_{i})^{T}A(\zz_{j}+q\yy_{j})\Big)$$
may be replaced with the integral
\begin{equation}\label{inty}
\int_{\substack{\yy_{1},\ldots,\yy_{m} \in \R^n:\\|\zz_{i} + q\yy_{i}| \leq P_{i}}}
e\Big(\sum_{1 \leq i \leq j \leq m}\beta_{ij}(\zz_{i}+q\yy_{i})^{T}A(\zz_{j}+q\yy_{j})\Big)\dy_{1}\cdots\dy_{m},
\end{equation}
with  error
$$\ll \max_{\substack{1 \leq s \leq m\\1 \leq t \leq n}}\Big|\frac{\partial}{\partial y_{st}}e\Big(\sum_{1 \leq i \leq j \leq m}\beta_{ij}(\zz_{i}+q\yy_{i})^{T}A(\zz_{j}+q\yy_{j})\Big)\Big|\frac{\Pi^{n}}{q^{mn}} +
\frac{\Pi^{n}}{q^{mn-1}P_{1}}.$$

For any $1 \leq s \leq m, 1 \leq t \leq n$, on substituting
$\mathbf{u}_i = \mathbf{z}_i + q\mathbf{y}_i$ we find that
\begin{align*}
  & \frac{\partial}{\partial \mathbf{y}_s}
  \sum_{1 \leq i \leq j \leq m}
  \beta_{ij}(\zz_{i}+q\yy_{i})^{T}A(\zz_{j}+q\yy_{j})\\
  = & q \frac{\partial}{\partial \mathbf{u}_s}
  \sum_{1 \le i \le j \le m} \beta_{ij} \mathbf{u}_i^T A \mathbf{u}_j\\
  = & q \left(2 \beta_{ss} A \mathbf{u}_s
  + \sum_{1 \le i < s} \beta_{is} A \mathbf{u}_i
  + \sum_{s<j\le m} \beta_{sj} A \mathbf{u}_j \right),
\end{align*}
whence \eqref{alphaMaq} and $|\mathbf{u}_i| \le P_i \;
(1 \le i \le m)$ yield
\begin{align*}
& \frac{\partial}{\partial y_{st}}e\Big(\sum_{1 \leq i \leq j \leq m}
\beta_{ij}(\zz_{i}+q\yy_{i})^{T}A(\zz_{j}+q\yy_{j})\Big)\\
\ll &  q\Big(\sum_{1 \leq i \leq s}|\beta_{is}|P_{i} +
\sum_{s<j \le m}|\beta_{sj}|P_{j}\Big)\\
\ll & q\Pi^{\Delta}P_{s}^{-1},
\end{align*}
and therefore the difference between the sum and the integral is
$$O\Big(\frac{\Pi^{n+\Delta}}{q^{mn-1}P_{1}}\Big).$$

Making the change of variables
$\vv_{i} = P_i^{-1} (\zz_{i} + q\yy_{i}) \;
(1 \leq i \leq m)$, in \eqref{inty}, we see that
$$\int_{\substack{\yy_{1},\ldots,\yy_{m} \in \R^n:\\|\zz_{i} + q\yy_{i}| \leq P_{i}}}
e\Big(\sum_{1 \leq i \leq j \leq m}\beta_{ij}(\zz_{i}+q\yy_{i})^{T}A(\zz_{j}+q\yy_{j})\Big)\dy_{1}\cdots\dy_{m} =
q^{-mn}\Pi^{n}I(\textbf{P},\bm{\beta}).$$
The lemma now follows easily on using the trivial bound
$|S_{\mathbf{a}, q}(\mathbf{b})| \le q^{mn}$ and $q \ll \Pi^\Delta$.
\end{proof}

For $Q \geq 1$ define
\begin{equation}\label{SR}
\mathfrak{S}(Q;\bb) := 
\sum_{q \ll Q}q^{-mn}\sum_{\substack{\textbf{a}\,(\bmod{q})\\(\textbf{a},q)=1}}S_{\textbf{a},q}(\bb),
\end{equation}
and
\begin{equation}\label{IR}
\mathfrak{I}(Q; \cc) := \int_{|\bm{\eta}| \ll Q}I(\bm{\eta})e\Big(-\sum_{1 \leq i \leq j \leq m}\eta_{ij}c_{ij}\Big)\deta,
\end{equation}
for $\cc =(c_{ij})_{1 \leq i \leq j \leq m} \in \mathbb{R}^{R}$.

\begin{lemma}\label{NABequals}
Let $0 < \Delta < \frac{1}{\gamma(2R+3)}$.  Then there exists $\delta > 0$ such that
\begin{equation*}
N(A,B) = \Pi^{n-m-1}\mathfrak{S}(\Pi^{\Delta};\bb)
\mathfrak{I}(\Pi^{\Delta};\cc) + O\Big(\Pi^{n-m-1-\delta}\Big),
\end{equation*}
where $c_{ij} = (P_{i}P_{j})^{-1}B_{ij}$ for $1 \leq i \leq j \leq m$.
\end{lemma}

\begin{proof}
By our assumption on $\Delta$,
the observation \eqref{nab1} and Lemma \ref{Disjoint},
for some $\delta > 0$ we have 
\begin{align*}
N(A,B) = & \int_{\M'(\Delta)}S(\bm{\alpha},\bb)\dalpha + O\Big(\Pi^{n-m-1-\delta}\Big)\\
= &
\sum_{q \ll \Pi^{\Delta}}
\sum_{\substack{\textbf{a}\,(\bmod{q})\\(\textbf{a},q)=1}}\int_{\M'_{\textbf{a},q}(\Delta)} S(\bm{\alpha},\bb)\dalpha +
O\Big(\Pi^{n-m-1-\delta}\Big).
\end{align*}
We shall use Lemma \ref{Salpha} to approximate $S(\bm{\alpha},\bb)$ on $\M'_{\textbf{a},q}(\Delta)$.  The error term,
when integrated over $\M'(\Delta)$, is bounded by
\begin{align*}
&\ll \sum_{q \ll \Pi^{\Delta}}
\sum_{\substack{\textbf{a}\,(\bmod{q})\\(\textbf{a},q)=1}}\int_{|\beta_{ij}| \ll \Pi^{\Delta}
(P_{i}P_{j})^{-1}}\Pi^{n + 2\Delta}P_{1}^{-1}\dbeta\\
& \ll \sum_{q \ll \Pi^{\Delta}}
\sum_{\substack{\textbf{a}\,(\bmod{q})\\(\textbf{a},q)=1}}
\Pi^{n+2\Delta+R\Delta-(m+1)} P_1^{-1}\\
& \ll \Pi^{n-m-1 + \Delta(2R+3)}P_{1}^{-1}\\
& \ll \Pi^{n-m-1 - \delta'},
\end{align*} 
for some $\delta' > 0$, by \eqref{gamma_useful}, $\gamma_1=1$ and
our assumption on $\Delta$.  This contributes to the error term in the lemma.
The main term gives
$$N(A,B) =
\Pi^n
\mathfrak{S}(\Pi^{\Delta};\bb)
\int_{\bm{\beta}}I(\textbf{P},\bm{\beta})e\Big(-\sum_{1 \leq i \leq j \leq m}\beta_{ij}B_{ij}\Big)\dbeta,$$
where the integral is over $|\beta_{ij}| \ll \Pi^{\Delta}(P_{i}P_{j})^{-1}
\; (1 \leq i \leq j \leq m)$.
Substituting $\eta_{ij} = P_{i}P_{j}\beta_{ij}$ completes the proof of the lemma.
\end{proof}

Let
\[
  \mathfrak{S}(\bb) := \mathfrak{S}(\infty;\bb)
\]
be the singular series and
\[
  \mathfrak{I}(\cc) := \mathfrak{I}(\infty ; \cc)
\]
be the singular integral, with $\cc$ as in Lemma \ref{NABequals}.

\begin{lemma}\label{SRconv}
Assume that $n > (2\gamma + m(m-1))(R+1)$.  Then $\mathfrak{S}(\bb)$ is absolutely convergent.  Moreover, we have
$$|\mathfrak{S}(\bb) - \mathfrak{S}(\Pi^{\Delta};\bb)|
\ll \Pi^{-2\delta + \eps},$$
for some $\delta >0$, uniformly in $\bb$.
\end{lemma}

\begin{proof}
Let $(q, \mathbf{a})=1$.
We may use Lemma \ref{Weyl2}, with $P_{1} = \cdots = P_{m} = q$, $\gamma = m$, and $\alpha_{ij} = \frac{a_{ij}}{q}$ 
for $1 \leq i \leq j \leq m$, to see that either
\begin{equation}\label{Saqtheta}
|S_{\textbf{a},q}(\bb)| \ll q^{mn - \frac{n\Delta}{m+1}+\eps},\end{equation}
or that $(\frac{a_{ij}}{q})_{1 \leq i \leq j \leq m} \in \M(\Delta)$, for any $0 < \Delta < \frac{m+1}{2}$.
The latter option says that there exist coprime integers $q', (a'_{ij})_{1 \leq i \leq j \leq m}$
which satisfy
$$q' \ll q^{m\Delta}, \quad |q'a_{ij} - qa'_{ij}| \ll q^{m\Delta - 1}
\quad (1 \le i \le j \le m).$$
Since $(q, \mathbf{a})=1$,
these conditions are clearly impossible to satisfy
if $\Delta < \frac{1}{m}$, and therefore
\eqref{Saqtheta} must hold.  Setting $\Delta = \frac{1}{m} - \eps$ for any $\eps >0$ gives
\begin{equation}\label{Saqbound}
|S_{\textbf{a},q}(\bb)| \ll q^{mn - \frac{n}{2R}+\eps}.\end{equation}

Therefore we have
\[
  \sum_{q=1}^\infty q^{-mn}
  \sum_{\substack{\textbf{a}\,(\bmod{q})\\(\textbf{a},q)=1}}
  \left| S_{\mathbf{a}, q}(\mathbf{b}) \right| 
  \ll \sum_{q=1}^{\infty}q^{R-\frac{n}{2R}+\eps}
  \ll  \sum_{q=1}^{\infty}q^{-1-\frac{1}{2R}+\eps} \ll  1,
\]
since $n \geq (2\gamma + m(m-1))(R+1) + 1 \geq 2R(R+1) + 1$,
showing that $\mathfrak{S}(\bb)$ is absolutely convergent.

For the second part of the lemma, let $\delta > 0$ be as in \eqref{deltadef}.  Then we use \eqref{Saqbound} to see that
\begin{align*}|\mathfrak{S}(\bb) - \mathfrak{S}(\Pi^{\Delta};\bb)|
\ll & \sum_{q \gg \Pi^{\Delta}}q^{R-\frac{n}{2R}+\eps}\\
\ll & \sum_{q \gg \Pi^{\Delta}}q^{R-\frac{n}{2\gamma+m(m-1)}+\eps}\\
\ll & \sum_{q \gg \Pi^{\Delta}}q^{-\frac{2\delta }{\Delta}-1+\eps}\\
\ll & \Pi^{-2\delta + \eps}.
\end{align*}
\end{proof}

\begin{lemma}\label{Ietabound}
We have
$$I(\bm{\eta}) \ll \min\{1, \max|\eta_{ij}|^{-\frac{n}{2R}+\eps}\}.$$
\end{lemma}
\begin{proof}
The first bound here is trivial.  We may therefore assume that $\max|\eta_{ij}| > 1$.

Let $P \geq 1$ be a parameter, and define 
$$S'(\bm{\alpha}) := \sum_{|\xx_{1}|,\ldots,|\xx_{m}| \leq P}
e\Big(\sum_{1 \leq i \leq j \leq m}\alpha_{ij}\xx_{i}^{T}A\xx_{j}\Big).$$
By following the proof of Lemma \ref{Salpha} with $q=1$, $\textbf{a} = \textbf{0}$, we get
\begin{align*}
  S'(\bm{\alpha}) = &
  P^{mn}\int_{[-1,1]^{mn}}e
  \Big(P^{2}\sum_{1 \leq i \leq j \leq m}\alpha_{ij}
  \vv_{i}^{T}A\vv_{j}\Big)\dv_{1}\cdots\dv_{m}\\
  +& O\Big(\sum_{1 \leq i \leq j \leq m}|
  \alpha_{ij}|P^{mn+1} + P^{mn-1}\Big).
\end{align*}
It is a simple corollary to Lemma \ref{Weyl2} (see the corollary to Birch \cite[Lemma 4.3]{Birch} for instance) 
that for $\max|\alpha_{ij}| < P^{-1}$, we have
\begin{equation}
\label{depris}
S'(\bm{\alpha}) \ll
P^{mn+\eps}(P^{2}\max|\alpha_{ij}|)^{-\frac{n}{2R}}:
\end{equation}
If $\max|\alpha_{ij}| \le P^{-2}$, then the bound is trivial,
otherwise write $\max|\alpha_{ij}| = P^{-2+m\Delta}$ for a
suitable $\Delta>0$. Since $\max|\alpha_{ij}|<P^{-1}$,
on noting that $\gamma=m$ in our situation, we find that
$\Delta<\frac{1}{\gamma}$. Thus, by Lemma \ref{Disjoint},
the major arcs are disjoint, so $\bm{\alpha}$ is at the boundary of
$\M(\Delta)$ and consequently, for any $\eps>0$, we have
$\bm{\alpha} \not \in \M(\Delta+\eps)$. Hence by Lemma
\ref{Weyl2} we have
\[
  S'(\bm{\alpha}) \ll P^{mn(1-\frac{\Delta+\eps}{2R}+\eps)}
  \ll P^{mn+\eps} (P^2 \max\{|\alpha_{ij}|\})^
  {-\frac{n(\Delta+\eps)}{2\Delta R}},
\]
confirming the bound \eqref{depris}.
Therefore, whenever $\max|\alpha_{ij}| < P^{-1}$, we have
\begin{align*}
 & P^{mn}\int_{[-1,1]^{mn}}e\Big(P^{2}\sum_{1 \leq i \leq j \leq m}
  \alpha_{ij}\vv_{i}^{T}A\vv_{j}\Big)\dv_{1}\cdots\dv_{m}\\
 & \ll\Big(P^{2}\sum_{1 \leq i \leq j \leq m}|
   \alpha_{ij}| + 1\Big)P^{mn-1}\\
 & + P^{mn+\eps}(P^{2}\max|\alpha_{ij}|)^{-\frac{n}{2R}}.
\end{align*}
Substituting $\eta_{ij} = P^{2}\alpha_{ij}$, noting that the
left-hand side of the previous inequality is just
$P^{mn} I(\bm{\eta})$, we obtain
$$I(\bm{\eta}) \ll
\Big(\max|\eta_{ij}|+1\Big)P^{-1} + P^{\eps}(\max|\eta_{ij}|)^{-\frac{n}{2R}}.$$
For given $\bm{\eta}$,
we may set $P=\max |\eta_{ij}|^{1 + \frac{n}{2R}} > 1$,
this way defining $\alpha_{ij}$ as
$\alpha_{ij}=\eta_{ij}P^{-2}$,
which implies that $\max |\alpha_{ij}|< P^{-1}$.
Hence the bound above gives
$$I(\bm{\eta}) \ll \max|\eta_{ij}|^{-\frac{n}{2R}+\eps}.$$
\end{proof}

\begin{lemma}\label{IRconv}
Assume that $n > (2\gamma + m(m-1))(R+1)$.  Then $\mathfrak{I}(\cc)$
converges absolutely, and furthermore, for any $Q \geq 1$, we have
$$|\mathfrak{I}(\cc) - \mathfrak{I}(Q,\cc)| \ll Q^{-1+\eps},$$
uniformly in $\cc$.
\end{lemma}

\begin{proof}
Let $N := \max|\eta_{ij}|$.  Then for any $1 \ll Q_{1} < Q_{2}$,
for suitable positive constants $c_6$ and $c_7$ we have
\begin{align*}
|\mathfrak{I}(Q_{2},\cc) - \mathfrak{I}(Q_{1},\cc)| = &
\int_{c_6 Q_{1} \leq N \leq c_7 Q_{2}}I(\bm{\eta})
e\Big(-\sum_{1 \leq i \leq j \leq m}\eta_{ij}c_{ij}\Big)\deta\\
\ll & \int_{c_6 Q_{1} \leq N \leq c_7 Q_{2}}
\min\{1,N^{-R-1-\frac{1}{2R}+\eps}\}\deta,
\end{align*}
on using Lemma \ref{Ietabound} and since $n \geq (2\gamma + m(m-1))(R+1) + 1 \geq 2R(R+1) + 1$.
Therefore,
applying Fubini's Theorem, and noting that $Q_1 \gg 1$, we obtain
$$|\mathfrak{I}(Q_{2},\cc) - \mathfrak{I}(Q_{1},\cc)| \ll
\int_{c_6 Q_{1}}^{c_7 Q_{2}}N^{-2-\frac{1}{2R}+\eps}d
N \ll Q_1^{-1-\frac{1}{2R}+\eps}.$$
Both parts of the lemma now follow.
\end{proof}

We now make use of our assumption that $B$ is Minkowski reduced:
As is well known, this implies that
\[
  \det{B} \ll \prod_{i=1}^m B_{ii} \ll \det{B},
\]
where the implied $O$-constant depends only on the dimension $m$ of $B$.
Combining Lemmas \ref{NABequals}, \ref{SRconv} and \ref{IRconv},
and noting that $\mathfrak{S}(\bb) \ll 1$ and
$\mathfrak{I}(\mathbf{c}) \ll 1$
(see \S \ref{depression}, \S \ref{oktober},
\eqref{ungef} and \eqref{alphainfty}) we therefore obtain the following.
\begin{lemma}
\label{geburtstag}
Assume that $n > (2\gamma + m(m-1))(R+1)$.  Then there exists $\delta > 0$ such that
$$N(A,B) = \Pi^{n-m-1}\mathfrak{S}(\bb)\mathfrak{I}(\cc) +
O\Big((\det{B})^{\frac{n-m-1}{2}-\delta}\Big),$$
where $c_{ij} = (P_{i}P_{j})^{-1}B_{ij}$ for $1 \leq i \leq j \leq m$.
\end{lemma}

\section{Singular Series}
\label{depression}
The singular series $\mathfrak{S}(\bb)$ corresponds to $p$-adic solutions to the system of equations, 
and we shall show that it factors as a product over all primes of $\alpha_{p}(A,B)$.

\begin{lemma}
\label{weihnachten}
Suppose that $n > (2\gamma+m(m-1))(R+1)$.
Then we have
\[
  \mathfrak{S}(\bb) = \prod_{p}\alpha_{p}(A,B).
\]
\end{lemma}

\begin{proof}
Since $\mathfrak{S}(\mathbf{b})$ is absolutely convergent by
Lemma \ref{SRconv},
a standard argument (see Birch \cite[Section 7]{Birch} for example)
then gives
\begin{align*}
\mathfrak{S}(\bb) = & \prod_{p}\sum_{r=0}^{\infty}
\sum_{\substack{\textbf{a}\,(\bmod{p^{r}})\\(\textbf{a},p)=1}}p^{-rmn}S_{\textbf{a},p^{r}}(\bb)\\
= & \prod_{p}\mathfrak{S}_{p}(\bb),
\end{align*}
say.  Now, for each prime $p$, we have
\begin{align*}\mathfrak{S}_{p}(\bb)
= &\lim_{N \rightarrow \infty}\sum_{r=0}^{N}
\sum_{\substack{\textbf{a}\,(\bmod{p^{r}})\\
(\textbf{a},p)=1}}p^{-rmn}S_{\textbf{a},p^{r}}(\bb)\\
= &\lim_{N \rightarrow \infty}(p^{-N})^{mn-R}\\
\times &
\#\{\xx_{1},\ldots, \xx_{m} \pmod{p^{N}}: 
\xx_{i}^{T}A\xx_{j} \equiv B_{ij} \pmod{p^{N}}
\; (1 \le i \le j \le m)\}\\
= &\lim_{N \rightarrow \infty}(p^{-N})^{mn-R}\#\{X \pmod{p^{N}}:
X^{T}AX \equiv B \pmod{p^{N}}\}.
\end{align*}
By \cite[Lemma 5.6.1]{Kitaokabook}, there exists an integer $t\geq 0$
such that
$$(p^{-N})^{mn-R}\#\{X \pmod{p^{N}}: X^{T}AX \equiv B \pmod{p^{N}}\}$$
remains constant for all $N \geq t$.  This is $\alpha_{p}(A,B)$ defined in \eqref{alphap}, and so
we have $\mathfrak{S}_{p}(\bb) = \alpha_{p}(A,B)$.
\end{proof}

\section{Singular Integral}
\label{oktober}
The proof of Theorem \ref{thm} will be complete on showing that
$\Pi^{n-m-1}\mathfrak{I}(\cc) = \alpha_{\infty}(A,B)$, defined in \eqref{alphainfty}.

Let $U \subset \mathbb{R}^{m(m+1)/2}$ be a real neighbourhood of $B$.  Let $V \subset \mathbb{R}^{mn}$ be the set of real
$n \times m$ matrices $X$ such that $X^{T}AX$ lies inside $U$.  Then it is known 
(see \cite[Chapter A.3]{Cassels} for instance) that $\alpha_{\infty}(A,B)$ is equal to the limit of
$$\frac{\text{vol}(V)}{\text{vol}(U)}$$
as the neighbourhood $U$ shrinks to $B$.
Therefore, taking the neighbourhood
$$\prod_{1 \leq i \leq j \leq m} [B_{ij} - \eps P_{i}P_{j}, B_{ij} + \eps P_{i}P_{j}],$$
for $\eps > 0$, we may deduce that
\begin{equation}
\label{woche10}
\alpha_{\infty}(A,B) = \lim_{\eps \rightarrow 0}\frac{1}{\Pi^{m+1}(2\eps)^{R}}
\int_{\substack{|\xx_{i}^{T}A\xx_{j} - B_{ij}| \leq (P_{i}P_{j})\eps\\1 \leq i \leq j \leq m}}d\xx_{1}\cdots d\xx_{m}.
\end{equation}

For  $\cc=(c_{ij})_{1 \leq i \leq j \leq m}$ with $c_{ij} = (P_{i}P_{j})^{-1}B_{ij}$, 
let $V(\cc)$ denote the real variety defined by
$$\xx_{i}^{T}A\xx_{j} - c_{ij} = 0 \quad (1 \leq i \leq j \leq m).$$
\begin{lemma}\label{nonsing}
The variety  $V(\cc)$ is non-empty and non-singular.
\end{lemma}
\begin{proof}
By our choice of the $P_i$ in \eqref{holiday},
there exist real vectors $\yy_{1},\ldots,\yy_{m}$ such that
$$\yy_{i}^{T}A\yy_{j} = B_{ij} \quad (1\leq i \leq j \leq m).$$
Therefore taking $\xx_{i} = P_{i}^{-1}\yy_{i}$ for each $i \in \{1,\ldots, m\}$ gives a real point on $V(\cc)$.

Now consider the Jacobian matrix of this variety.  This is an $R \times mn$ matrix, and suppose that there exist real
vectors $\xx_{1},\ldots,\xx_{m}$ where this Jacobian has rank strictly less than $R$.  Therefore the rows of the Jacobian are
linearly dependent,
and considering the $n$ columns corresponding to
some suitable vector $\xx_{i}$, we deduce that
there exist real numbers $\lambda_{1},\ldots,\lambda_{m}$, not all zero, such that
\[
  A (2\lambda_{i}\xx_{i} + \sum_{\substack{j=1\\j \ne i}}^m
  \lambda_j \xx_j ) = \mathbf{0}.
\]
Since $A$ is non-singular, we must have
\[
  2\lambda_{i}\xx_{i} + \sum_{\substack{j=1\\j \ne i}}^m
  \lambda_j \xx_j = \mathbf{0}.
\]
Therefore the matrix $X$ whose columns are the vectors $\xx_{1},\ldots,\xx_{m}$ does not have full rank. It follows that
$X^{T}AX$ does not have full rank for any vectors $\xx_{1},\ldots,\xx_{m}$ where the Jacobian does not have full rank. 
Now the matrix $B$ has full rank, and the matrix
$C = (c_{ij})_{1 \leq i,j \leq m}$ can be written as $C=DBD$, where $D$
denotes the diagonal matrix having entries $P_1^{-1}, \ldots, P_m^{-1}$
on the diagonal. Therefore, also $C$ has full rank, whence we have shown
that
there cannot be a solution to $X^{T}AX = C$ where
$X$ does not have full rank. This shows that the variety $V(\cc)$ is
non-singular.
\end{proof}

Combining Lemma \ref{geburtstag}, Lemma \ref{weihnachten},
\eqref{woche10} and the following lemma we conclude
the proof of Theorem \ref{thm}.

\begin{lemma}
We have
$$\Pi^{n-m-1}\mathfrak{I}(\cc) = 
\lim_{\eps \rightarrow 0}\frac{1}{\Pi^{m+1}(2\eps)^{R}}
\int_{\substack{|\xx_{i}^{T}A\xx_{j} - B_{ij}| \leq (P_{i}P_{j})\eps\\1 \leq i \leq j \leq m}}d\xx_{1}\cdots d\xx_{m}.$$
\end{lemma}

\begin{proof}

For simplicity in notation, we shall denote the variety $V(\cc)$ by
$$G_{i,j}(\xx) = 0 \quad (1 \leq i \leq j \leq m),$$
for $\xx = (\xx_{1},\ldots,\xx_{m}) \in \mathbb{R}^{mn}$.
By Lemma \ref{nonsing}, this variety is non-empty and non-singular.
Therefore the variety has positive $(mn-R)$-dimensional measure, and 
the Jacobian matrix
$$\Big(\frac{\partial G_{i,j}(\xx)}{\partial x_{st}}\Big)_
{\substack{1 \leq i \leq j \leq m\\1 \leq s \leq m, 1 \leq t \leq n}}$$
has rank $R$ at all real points.
Since $A$ is positive definite,
for any $\eps >0$, the set $V(\mathbf{c}, \eps)$ of real $\xx$ satisfying
$$|G_{i,j}(\xx)| \leq \eps \quad (1 \leq i \leq j \leq m)$$
is closed and bounded, and hence compact. 
Moreover, by continuity, for small enough $\eps$
the Jacobian is still non-singular at
any point of this set, because this is true for $V(\mathbf{c})$.
Therefore we can partition $V(\cc, \varepsilon)$ into a finite number of
measurable partitions, and
on each partition, say $\xi$, there exists some $R$-tuple
$x_{s_{1}t_{1}}, \ldots,
x_{s_{R}t_{R}}$ with $1 \leq s_{1},\ldots,s_{R} \leq m, 1 \leq t_{1},\ldots,t_{R}\leq n$, such that for
\[
  \delta := \det\Big(\frac{\partial G_{i,j}(\xx)}
  {\partial x_{s_{k}t_{k}}}\Big)_{\substack{1 \leq i \leq j \leq m\\
  1 \leq k \leq R}}
\]
we have
\[
  |\delta| \gg 1
\]
for all points in $\xi$.
In particular, for all $1 \le k \le R$ 
for at least one pair $i,j$ we have
\begin{equation}
\label{phd_comics}
 \left| \frac{\partial G_{i,j}(\yy,\zz)}{\partial x_{s_{k}t_{k}}} \right|
 \gg 1
\end{equation}
throughout $\xi$, with an implied constant independent
of $\eps$.
Since the number of possibilities to choose the $s_k$ and the $t_k$ is
finite and independent of $\eps$, we can assume that the number of
partitions is independent of $\eps$ as well.

We shall write a typical vector $\xx = (\xx_{1},\ldots,\xx_{m})\in \xi$ as
$(\yy,\zz)$, where 
$$\yy = (x_{s_{1}t_{1}}, \ldots, x_{s_{R}t_{R}}),$$ and $\zz$ denotes the remaining variables.  
Suppose that $(\yy^{(1)},\zz)$ is a
point in $\xi$ which lies on the variety $V(\cc)$.  Then we have
$$|G_{i,j}(\yy,\zz) - G_{i,j}(\yy^{(1)},\zz)| \leq \eps \quad (1 \leq i \leq j \leq m).$$
By \eqref{phd_comics} and the mean-value theorem, it follows that 
$|x_{s_{k}t_{k}} - x_{s_{k}t_{k}}^{(1)}| \ll \eps$
for each $1 \leq k \leq R$.

Now we may write, by Taylor's Theorem,
\begin{align*}
& G_{i,j}(\yy,\zz) - G_{i,j}(\yy^{(1)},\zz)\\ & = \sum_{k = 1}^{R}
(x_{s_{k}t_{k}}-x_{s_{k}t_{k}}^{(1)}) \frac{\partial G_{i,j}(\yy^{(1)},\zz)}{\partial x_{s_{k}t_{k}}} + 
O(\eps^{2}) \quad (1 \leq i \leq j \leq m),
\end{align*}
on noting that
the second partial derivatives of the $G_{i,j}$ are all constant.
Therefore, inverting these $R$ linear equations,
we see that the conditions $|G_{i,j}(\xx)| \leq \eps
\; (1 \leq i \leq j \leq m)$ imply that 
$\yy$ lies in a region of volume $(2\eps)^{R}\delta^{-1} + O(\eps^{R+1})$.

Hence we have
$$\frac{1}{(2\eps)^{R}}\int_{\xi}d\xx_{1}\cdots d\xx_{m} = 
\int_{V(\cc)\cap \xi}\frac{d\zz}{\delta} + O(\eps).$$
On summing over all partitions $\xi$ and taking the limit as $\eps \rightarrow 0$, the right side above
is $\mathfrak{I}(\cc)$, following the argument of \cite[Section 6]{Birch}.
The left side becomes
\begin{align*}
& \lim_{\eps \rightarrow 0}\frac{1}{(2\eps)^{R}}\int_{\substack{|G_{i,j}(\xx)| \leq \eps\\1 \leq i \leq j \leq m}}
d\xx_{1}\cdots d\xx_{m}\\ & = 
\Pi^{-n}\lim_{\eps \rightarrow 0}\frac{1}{(2\eps)^{R}}
\int_{\substack{|\xx_{i}^{T}A\xx_{j} - B_{ij}| \leq (P_{i}P_{j})\eps\\1 \leq i \leq j \leq m}}d\xx_{1}\cdots d\xx_{m},
\end{align*}
on making a change of variables.  This completes the proof of the lemma.
\end{proof}

\providecommand{\bysame}{\leavevmode\hbox to3em{\hrulefill}\thinspace}
\providecommand{\MR}{\relax\ifhmode\unskip\space\fi MR }
\providecommand{\MRhref}[2]{%
  \href{http://www.ams.org/mathscinet-getitem?mr=#1}{#2}
}
\providecommand{\href}[2]{#2}

\end{document}